\documentclass[10pt]{amsart}
\usepackage{amsmath}
\usepackage[usenames,dvipsnames]{color}
\usepackage{parskip}
\usepackage{amsfonts}
\usepackage{amscd}
\usepackage[centertags]{amsmath}
\usepackage{amssymb}
\usepackage{stmaryrd}
\usepackage[all,cmtip]{xy}
\usepackage[english]{babel}
\usepackage{tabularx}
\usepackage{amsxtra}
\usepackage{euscript}
\usepackage[T1]{fontenc}
\usepackage{doc, exscale, fontenc, latexsym, syntonly}
\usepackage{amsfonts}
\usepackage{amsthm}
\usepackage{graphicx}
\usepackage{mciteplus}
\usepackage{cite}

\newtheorem{theorem}{Theorem}[section]
\newtheorem{corollary}[theorem]{Corollary}
\newtheorem{lemma}[theorem]{Lemma}
\newtheorem{proposition}[theorem]{Proposition}
\newtheorem{remark}[theorem]{Remark}
\theoremstyle{definition}
\newtheorem{definition}[theorem]{Definition}
\newtheorem{example}[theorem]{Example}
\theoremstyle{remark}

\numberwithin{figure}{section}
\numberwithin{table}{section}

\newcommand*\acknowledgment[1]{%
	\begingroup\noindent
	\rightskip\leftskip
	\begin{flushleft}\textbf{\large Acknowledgment.}\, #1%
		\par\vspace*{1mm}\end{flushleft}\endgroup}
\begin{document}

\title[TOPOLOGICAL GROUP CONSTRUCTION IN PROXIMITY AND DESCRIPTIVE PROXIMITY SPACES]{TOPOLOGICAL GROUPS IN PROXIMITY AND DESCRIPTIVE PROXIMITY SPACES}

\author{MEL\.{I}H \.{I}S}
\date{\today}

\address{\textsc{Melih Is,}
Ege University\\
Faculty of Sciences\\
Department of Mathematics\\
Izmir, Turkey}
\email{melih.is@ege.edu.tr}



\begin{abstract}
 This paper aims to examine the version of the topological group structure in proximity and especially descriptive proximity spaces, that is, the concepts of proximal group and descriptive proximal group are introduced. In addition, the concepts of homomorphism and isomorphism, which give important results in group theory, are discussed by interpreting the concepts of continuity in the theory of (descriptive) proximity.
\end{abstract}

\maketitle

\section{Introduction}
\label{intro}
\quad Topology is concerned with the study of properties preserved under continuous transformations, capturing the concept of nearness between elements of a set. Over the years, various approaches to topological spaces have been explored, each offering unique perspectives on the fundamental notions of continuity and proximity \cite{Riesz:1908,Wallace:1941,Smirnov:1952,Efremovic:1952,Lodato:1964,Cech:1966,NaimpallyPeters:2013}. One such approach that has gained significant attention is nearness theory, which provides an alternative framework for analyzing topological structures through the concept of descriptive proximity \cite{Peters:2013}.

\quad In nearness theory, the traditional notion of open sets is replaced by a more intuitive concept of near sets, characterized by a binary relation that describes the qualitative closeness between elements in a set. This approach introduces the notion of proximity spaces, which generalize the concept of metric spaces and provide a deeper understanding of the relationships between points based on qualitative descriptions rather than precise distances. Furthermore, in descriptive proximity spaces, proximity relations are tailored to specific features or characteristics, making them particularly suitable for applications in areas such as data analysis and pattern recognition \cite{NaimpallyPeters:2013,Peters:2014,Peters:2017}.

\quad The aim of this article is to explore the construction of topological groups within the context of nearness theory, with a specific focus on proximity and descriptive proximity spaces. As a slightly different concept, the studies \cite{MaheswariGnana:2022,MaheswariGnana:2023,InanUckun:2023} combine the ideas of topological space and near groups in order to define topological (or semitopological) near groups on a nearness approximation space and investigate their features such as group homomorphism of these groups. Topological groups, which combine algebraic and topological structures, offer a natural setting for studying the interplay between group operations and continuous mappings \cite{Husain:1966}. By leveraging the concepts of proximity and descriptive proximity, we seek to investigate the topological properties of these groups and the implications they hold for the overall structure of the underlying space.

\quad In Section \ref{sec:1}, we provide a concise overview of the fundamental concepts and definitions in nearness theory, establishing the groundwork for the subsequent discussions. This includes introducing the concept of proximity relations and their axiomatic properties, as well as delving into the qualitative nature of descriptive proximity relations. The main part of this article is presenting the construction of topological groups in proximity and descriptive proximity spaces. We explore the compatibility of group operations with nearness structures, investigating the behavior of near sets under group multiplication and inversion. We also give interesting examples in terms of the properties of (descriptive) proximal groups. One of the important results explicitly investigates a homomorphism, or more strongly an isomorphism, between proximal groups. Furthermore, there are important implications about the proximal group setting of isomorphism theorems of groups. Section \ref{sec:3} is dedicated to introducing the concept of descriptive proximal groups. Along with exciting examples in this section, we clearly state that our investigation is not only of theoretical interest but also holds practical implications. Topological groups constructed within  descriptive proximity spaces have the potential to find applications in diverse fields, ranging from data analysis and pattern recognition to the study of social networks and cognitive sciences.

\quad In summary, this article endeavors to contribute to the burgeoning field of nearness theory by exploring the topological group construction within proximity and descriptive proximity spaces. By offering a fresh perspective on the interplay between group structures and nearness relations, we aim to enrich the understanding of topological properties in nearness-based settings and open up new avenues for future research. 

\section{Preliminaries}
\label{sec:1}
\quad In this section, we simply state informative facts about proximity and descriptive proximity spaces. These facts will be frequently used in Section \ref{sec:2} and \ref{sec:3}.

\quad We first start with presenting the definition of proximity spaces with respect to Lodato, \v{C}ech, and Efremovi\v{c}.

\begin{definition}\cite{Lodato:1964}
	Given a nonempty space $Y$, a relation $\delta$ on $2^{Y}$ is said to be a \textit{Lodato proximity} provided that the properties
	\begin{itemize}
		\item[\textbf{L1.}] $B_{1} \ \delta \ B_{2}$ implies $B_{2} \ \delta \ B_{1}$.
		\item[\textbf{L2.}] $B_{1} \ \delta \ B_{2}$ implies that $B_{1}$ and $B_{2}$ are nonempty.
		\item[\textbf{L3.}] That the intersection of $B_{1}$ and $B_{2}$ is nonempty implies $B_{1} \ \delta \ B_{2}$.
		\item[\textbf{L4.}] $B_{1} \ \delta \ (B_{2} \cup B_{3})$ if and only if $B_{1} \ \delta \ B_{2}$ or $B_{1} \ \delta \ B_{3}$.
		\item[\textbf{L5.}] For each $b_{2} \in B_{2}$, $B_{1} \ \delta \ B_{2}$ and $\{b_{2}\} \ \delta \ B_{3}$ imply that $B_{1} \ \delta \ B_{3}$.
	\end{itemize}
    hold for all subsets $B_{1}$, $B_{2}$, and $B_{3}$ in $Y$.
\end{definition}

\quad $B_{1} \ \delta \ B_{2}$ is interpreted as "$B_{1}$ is near $B_{2}$", whereas $B_{1} \ \underline{\delta} \ B_{2}$ is read as "$B_{1}$ is far from $B_{2}$". Another definiton of the proximity is given by E. \v{C}ech \cite{Cech:1966}: The relation $\delta$ on $2^{Y}$ is said to be a \textit{\v{C}ech proximity} if the properties $L1-L4$ hold. In addition, $\delta$ is called an \textit{Efremovi\v{c} Proximity} \cite{Efremovic:1952} provided that the properties of \v{C}ech proximity ($L1-L4$) satisfy and an extra condition
\begin{itemize}
	\item[\textbf{EF}] $B_{1} \ \underline{\delta} \ B_{2}$ implies that there exists a subset $K$ of $Y$ such that $B_{1} \ \underline{\delta} \ K$ and $(Y - K) \ \underline{\delta} \ B_{2}$
\end{itemize}
holds.

\quad In this paper, if we want to emphasize the Lodato proximity relation or \v{C}ech proximity relation, we use the expression L-proximity or C-proximity for short, respectively. Unless otherwise emphasized, the simple expression $\delta$ refers to Efremovi\v{c} proximity. Therefore, $(Y,\delta)$ is called a \textit{proximity space} and simply denoted by pspc. Efremovi\v{c} proximity is stronger than L-proximity or C-proximity. The \textit{discrete proximity} $(Y,\delta)$, one of the basic proximity examples, is given by $B_{1} \ \delta \ B_{2}$ if and only if $B_{1} \cap B_{2} \neq \emptyset$ for all $B_{1}$, $B_{2} \in 2^{Y}$ \cite{NaimpallyWarrack:1970}. 

\quad The set of all points in $Y$ that are near $B_{1}$, which is the set \[B_{1}^{\delta} = \{y \in Y : \{y\} \ \delta \ B_{1}\},\] is known as the closure of a subset $B_{1}$, indicated by the symbol cl$B_{1}$ \cite{NaimpallyWarrack:1970}. Mathematically, one has $B_{1}^{\delta} =$ cl$B_{1}$. Then, by considering Kuratowski closure axioms \cite{Kuratowski:1958}, a topology $\tau(\delta)$ can be associated with the pspc $(Y,\delta)$.

\quad In proximity spaces, the term continuity is defined using the proximity relation instead of open sets. Explicitly, a function $k : (Y_{1},\delta_{1}) \rightarrow (Y_{2},\delta_{2})$ between two pspcs is considered continuous (we generally say \textit{proximally continuous} and simply denoted by pcont) if it preserves proximity; that is, for any subsets $B_{1}$ and $B_{2}$ of $Y$, if $B_{1}$ is near $B_{2}$ with respect to $\delta_{1}$, then $k(B_{1})$ is near $k(B_{2})$ with respect to $\delta_{2}$ \cite{Efremovic:1952,Smirnov:1952}. If $k_{1}$ and $k_{2}$ are two pcont maps, then so is their composition $k_{1} \circ k_{2}$ \cite{NaimpallyWarrack:1970}. A pcont map $k : (Y_{1},\delta_{1}) \rightarrow (Y_{2},\delta_{2})$ is called a \textit{proximal isomorphism} if the inverse map $k^{-1} : (Y_{2},\delta_{2}) \rightarrow (Y_{1},\delta_{1})$ is also pcont \cite{NaimpallyWarrack:1970}.

\quad When one has two pspcs $(Y_{1},\delta_{1})$ and $(Y_{2},\delta_{2})$, it is possible to obtain a new pspc $(Y_{1} \times Y_{2},\delta)$ by the cartesian product of them. The \textit{cartesian product proximity relation} $\delta$ is given as follows \cite{Leader:1964}. For any $(B_{1} \times B_{2})$, $(C_{1} \times C_{2}) \in 2^{Y_{1} \times Y_{2}}$, $(B_{1} \times B_{2}) \ \delta \ (C_{1} \times C_{2})$ if and only if $B_{1} \ \delta_{1} \ C_{1}$ and $B_{2} \ \delta_{2} \ C_{2}$. Assume that $(Y,\delta)$ is a pspc and $V$ is a subset of $Y$. Another new proximity $\delta_{V}$, called an \textit{induced or subspace proximity}, is defined by $B_{1} \ \delta_{V} \ B_{2}$ if and only if $B_{1} \ \delta \ B_{2}$ for all $B_{1}$, $B_{2} \in 2^{V}$ \cite{NaimpallyWarrack:1970}. 

\quad The isomorphism theorems are fundamental results in group theory that describe the relationship between groups and their subgroups, as well as the structure of factor groups. They are also powerful tools in group theory and help us understand the structural aspects of groups, especially when dealing with homomorphisms and factor groups. They provide valuable insights into the relationship between groups and their quotients, which allows us to analyze the structures of groups more effectively. Recall that, in a group $(G_{1},\cdot)$ with a normal subgroup $N_{1} \subseteq G_{1}$, $G_{1}/N_{1}$ is defined as the set $\{aN_{1} : a \in G_{1}\}$.

\begin{theorem}\cite{Hungerford:1974} \label{teo3}
	\textbf{i)} Assume that $\beta : G_{1} \rightarrow H_{1}$ is a homomorphism of groups, the kernel of $\beta$, given by Ker$(\beta) = \{g_{1} \in G_{1} : \beta(g_{1}) = e_{H_{1}}\} \subseteq G_{1}$, is a normal subgroup, and $\beta(G_{1}) \subseteq H_{1}$ is a subgroup. Then $G_{1} / \text{Ker}(\beta)$ is isomorphic to Im$(\beta)$.
	
	\textbf{ii)} Given a group $G_{1}$, and subgroups $H_{1}$ and $N_{1}$ of $G$ with $N_{1}$ being a normal subgroup of $G_{1}$, we have that $H_{1}N_{1} \subseteq G_{1}$ is a subgroup, $N_{1} \subseteq H_{1}N_{1}$ is a normal subgroup, and the intersection $H_{1} \cap N_{1} \subseteq H_{1}$ is a normal subgroup. Then $H_{1}N_{1}/N_{1}$ is isomorphic to $H_{1}/(H_{1} \cap N_{1})$.  
	
	\textbf{iii)} Assume that $G_{1}$ is a group, and 
	$N_{1}$ and $K_{1}$ are normal subgroups of $G_{1}$ with $N_{1} \subseteq K_{1}$. Then $(G_{1}/N_{1}) / (K_{1}/N_{1})$ is isomorphic to $G_{1}/K_{1}$.
\end{theorem}

\quad In Theorem \ref{teo3}, $\textbf{i)}$, $\textbf{ii)}$, and $\textbf{iii)}$ are generally referred to \textit{First Isomorphism Theorem}, \textit{Second Isomorphism Theorem}, and \textit{Third Isomorphism Theorem}, respectively.

\quad A descriptive proximity relation, generally denoted by $\delta_{\Phi}$, on a nonempty set $Y$ is a binary relation that captures the concept of nearness or closeness between elements in the set \cite{Peters1:2007,Peters2:2007,Peters:2013}. It provides a qualitative way to compare how close or similar two elements are to each other, without involving precise distance measurements as in metric spaces.

\quad Consider the nonempty set $Y$ with any element (object) $y \in Y$. For each $j \in J$, $\phi_{j}$ is a function from $Y$ to real numbers and takes any element $y$ to the feature value of it. The set of probe functions is denoted by $\Phi = \displaystyle \{\phi_{j}\}_{j \in J}$. An object's description can be found in a feature vector $\Phi$. For any subsets $B_{1}$, $B_{2} \in 2^{Y}$, $B_{1} \ \delta_{\Phi} \ B_{2}$ if and only if $\Phi(B_{1}) \cap \Phi(B_{2}) \neq \emptyset$, where $\Phi(C_{1})$ is given by the sets $\{\Phi(c_{1}) : c_{1} \in C_{1}\}$. Here $B_{1} \ \delta_{\Phi} \ B_{2}$ means that $B_{1}$ is \textit{descriptively near} $B_{2}$ (similarly, $B_{1} \ \underline{\delta_{\Phi}} \ B_{2}$ is used to say $B_{1}$ is \textit{descriptively far from} $B_{2}$) and $\delta_{\Phi}$ is called a \textit{descriptive proximity relation} on the subsets of $Y$. The \textit{descriptive intersection} for the subsets $B_{1}$ and $B_{2}$ of $Y$ is defined by $\{b \in B_{1} \cup B_{2} : \Phi(b) \ \text{belongs to} \ \Phi(B_{1}) \cap \Phi(B_{2})\}$ and generally denoted by $\displaystyle B_{1} \bigcap_{\Phi} B_{2}$.

\begin{definition}\cite{DiGuadagPeterRaman:2018}
	Let $Y$ be a nonempty space. Then a relation $\delta_{\Phi}$ is said to be a \textit{descriptive Lodato proximity} provided that the properties
	\begin{itemize}
		\item[\textbf{DL1.}] $B_{1} \ \delta_{\Phi} \ B_{2}$ implies $B_{2} \ \delta_{\Phi} \ B_{1}$.
		\item[\textbf{DL2.}] $B_{1} \ \underline{\delta_{\phi}} \ \emptyset$ for all $B_{1}$ in $2^{Y}$.
		\item[\textbf{DL3.}] That the descriptive intersection of $B_{1}$ and $B_{2}$ is nonempty implies $B_{1} \ \delta_{\Phi} \ B_{2}$.
		\item[\textbf{DL4.}] $B_{1} \ \delta_{\Phi} \ (B_{2} \cup B_{3})$ if and only if $B_{1} \ \delta_{\Phi} \ B_{2}$ or $B_{1} \ \delta_{\Phi} \ B_{3}$.
		\item[\textbf{DL5.}] For each $b_{2} \in B_{2}$, $B_{1} \ \delta_{\Phi} \ B_{2}$ and $\{b_{2}\} \ \delta_{\Phi} \ B_{3}$ imply that $B_{1} \ \delta_{\Phi} \ B_{3}$.
	\end{itemize}
	hold for all subsets $B_{1}$, $B_{2}$, and $B_{3}$ in $Y$.
\end{definition}

\quad The relation $\delta_{\Phi}$ on $2^{Y}$ is said to be a \textit{descriptive Efremovi\v{c} proximity} if the properties $DL1-DL4$ hold, and in addition,
\begin{itemize}
	\item[\textbf{DEF}] $B_{1} \ \underline{\delta_{\Phi}} \ B_{2}$ implies that there exists a subset $K$ of $Y$ such that $B_{1} \ \underline{\delta_{\Phi}} \ K$ and $(Y - K) \ \underline{\delta_{\Phi}} \ B_{2}$
\end{itemize}
satisfies. 

\quad $(Y,\delta_{\Phi})$ is called a \textit{descriptive proximity space} and simply denoted by dpspc. A function $k : (Y_{1},\delta_{\Phi_{1}}) \rightarrow (Y_{2},\delta_{\Phi_{2}})$ between two dpspcs is considered continuous (we generally say \textit{descriptive proximally continuous} and simply denoted by dpcont) if it preserves descriptive proximity; that is, for any subsets $B_{1}$ and $B_{2}$ of $Y$, if $B_{1}$ is descriptively near $B_{2}$ with respect to $\delta_{\Phi_{1}}$, then $k(B_{1})$ is descriptively near $k(B_{2})$ with respect to $\delta_{\Phi_{2}}$ \cite{PetersTane:2021}. If $k_{1}$ and $k_{2}$ are two dpcont maps, then so is their composition. A dpcont map $k : (Y_{1},\delta_{\Phi_{1}}) \rightarrow (Y_{2},\delta_{\Phi_{2}})$ is called a \textit{descriptive proximal isomorphism} if the inverse map $k^{-1} : (Y_{2},\delta_{\Phi_{2}}) \rightarrow (Y_{1},\delta_{\Phi_{1}})$ is also dpcont \cite{PetersTane:2021}.

\quad Let $(Y_{1},\delta_{\Phi_{1}})$ and $(Y_{2},\delta_{\Phi_{2}})$ be any dpspcs. Then their cartesian product $Y_{1} \times Y_{2}$ admits a \textit{cartesian product descriptive proximity relation} $\delta_{\Phi}$ defined as follows \cite{PetersTane2:2022}. For any $(B_{1} \times B_{2})$, $(C_{1} \times C_{2}) \in 2^{Y_{1} \times Y_{2}}$, $(B_{1} \times B_{2}) \ \delta_{\Phi} \ (C_{1} \times C_{2})$ if and only if $B_{1} \ \delta_{\Phi_{1}} \ C_{1}$ and $B_{2} \ \delta_{\Phi_{2}} \ C_{2}$. Assume that $(Y,\delta_{\Phi})$ is a dpspc and $V$ is a subset of $Y$. A \textit{descriptive induced (or subspace) proximity}, denoted by $\delta_{\Phi_{V}}$ is defined by $B_{1} \ \delta_{\Phi_{V}} \ B_{2}$ if and only if $B_{1} \ \delta_{\Phi} \ B_{2}$ for all $B_{1}$, $B_{2} \in 2^{V}$. 

\quad Given two dpspcs, $(Y_{1},\delta^{1}_{\Phi})$ and $(Y_{2},\delta^{2}_{\Phi})$, the descriptive proximal mapping space $Y_{2}^{Y_{1}}$ is defined as the set $\{\beta : Y_{1} \rightarrow Y_{2} \ | \ \beta \ \text{is a dpcont-map}\}$ having the following descriptive proximity relation $\delta_{\Phi}$ on itself \cite{MelihKaraca:2023}: Let $B_{1}$, $B_{2} \subseteq Y$ and $\{\gamma_{j}\}_{j \in J}$ and $\{\gamma^{'}_{k}\}_{k \in K}$ be any subsets of dpcont maps in $Y_{2}^{Y_{1}}$. We say that $\{\gamma_{j}\}_{j \in J} \ \delta_{\Phi} \ \{\gamma^{'}_{k}\}_{k \in K}$ provided that $B_{1} \ \delta^{1}_{\Phi} \ B_{2}$ implies that $\gamma_{j}(B_{1}) \ \delta^{2}_{\Phi} \ \gamma^{'}_{k}(B_{2})$ for all $j$ and $k$.

\section{Proximal Groups}
\label{sec:2}

\begin{definition}\label{def1}
	Let $\delta$ and $\cdot$ be a proximity and a group operation on a set $G_{1}$, respectively. Then $(G_{1},\delta,\cdot)$ is said to be a proximal group when
	\begin{eqnarray*}
		\mu_{1} : G_{1} \times G_{1} \rightarrow G_{1},
	\end{eqnarray*} 
    defined by $\mu_{1}(g_{1},g_{1}^{'}) = g_{1}\cdot g_{1}^{'}$ for any $g_{1}$, $g_{1}^{'} \in G_{1}$,
    and
    \begin{eqnarray*}
    	\mu_{2} : G_{1} \rightarrow G_{1},
    \end{eqnarray*} 
    defined by $\mu_{2}(g_{1}) = g_{1}^{-1}$ for any $g_{1} \in G_{1}$, are pcont maps.
\end{definition}

\quad Recall that for any subsets $B_{1}$, $B_{2}$ and $B_{3}$ of a topological group $G_{1}$, $B_{1}^{-1}$ and $B_{2}\cdot B_{3}$ are given by $\{b_{1}^{-1} : b_{1} \in B_{1}\}$ and $\{b_{2}\cdot b_{3} \ | \ b_{2} \in B_{2}, \ b_{3} \in B_{3}\}$, respectively. 

\begin{example}\label{exm1}
	Consider $\mathbb{R}-\{0\}$ with a proximity $\delta$, defined by \[B_{1} \ \delta \ B_{2} \Leftrightarrow B_{1} \cap B_{2} \neq \emptyset\] for any subsets $B_{1}$, $B_{2}$ in $\mathbb{R}-\{0\}$, and the group operation $\cdot$ for the subsets of $\mathbb{R}-\{0\}$. Then we shall show that $(\mathbb{R}-\{0\},\delta,\cdot)$ is a proximal group. Define the maps
	\begin{eqnarray*}
		\mu_{1} : \mathbb{R}-\{0\} \times \mathbb{R}-\{0\} \rightarrow \mathbb{R}-\{0\} \hspace*{0.5cm} \text{and} \hspace*{0.5cm} \mu_{2} : \mathbb{R}-\{0\} \rightarrow \mathbb{R}-\{0\}
	\end{eqnarray*} 
	with $\mu_{1}(g_{1},g_{1}^{'}) = g_{1}\cdot g_{1}^{'}$ for any $g_{1}$, $g_{1}^{'} \in \mathbb{R}-\{0\}$, and, $\mu_{2}(g_{1}) = g_{1}^{-1}$ for any $g_{1} \in \mathbb{R}-\{0\}$, respectively. First, $\mu_{1}$ is pcont. Indeed, for any subsets $B_{1} \times B_{2}$, $C_{1} \times C_{2} \in \mathbb{R}-\{0\} \times \mathbb{R}-\{0\}$, the fact $B_{1} \times B_{2}$ is near $C_{1} \times C_{2}$ implies that $B_{1}$ is near $C_{1}$ and $B_{2}$ is near $C_{2}$. This means that $B_{1} \cap C_{1} \neq \emptyset$ and $B_{2} \cap C_{2} \neq \emptyset$, respectively. Therefore, there exist $x_{1}$, $x_{2} \in \mathbb{R}-\{0\}$ such that $x_{1} \in B_{1}$, $x_{1} \in C_{1}$, $x_{2} \in B_{2}$, and $x_{2} \in C_{2}$. Since $x_{1}\cdot x_{2}$ belongs to both $B_{1}\cdot B_{2}$ and $C_{1}\cdot C_{2}$, we find that $(B_{1}\cdot B_{2}) \cap (C_{1}\cdot C_{2}) \neq \emptyset$, which says that $(B_{1}\cdot B_{2}) \ \delta \ (C_{1}\cdot C_{2})$. Next, we claim that $\mu_{2}$ is pcont. Let $B_{1}$ and $B_{2}$ be any subsets of $\mathbb{R}-\{0\}$ such that $B_{1} \ \delta \ B_{2}$. Then $B_{1} \cap B_{2} \neq \emptyset$, that is, there exists $x_{1} \in \mathbb{R}-\{0\}$ such that $x_{1} \in B_{1}$ and $x_{1} \in B_{2}$. In a group $(\mathbb{R}-\{0\},\cdot)$, $x_{1}$ has an inverse $x_{1}^{-1}$. Moreover, $x_{1}^{-1}$ belongs to both $B_{1}^{-1}$ and $B_{2}^{-1}$. This shows that $B_{1}^{-1} \cap B_{2}^{-1} \neq \emptyset$. Thus, we observe that $B_{1}^{-1} \ \delta \ B_{2}^{-1}$. Consequently, $(\mathbb{R}-\{0\},\delta,\cdot)$ forms a proximal group.
\end{example}

\quad Note that $(\mathbb{R},\delta,+)$ is also a proximal group when we consider the usual addition $+$ on $\mathbb{R}$ as the group operation in Example \ref{exm1}, where $B_{1}+B_{2}$ is given by the set $\{b_{1}+b_{2} \ | \ b_{1} \in B_{1}, b_{2} \in B_{2}\}$.

\begin{example}
	Let $G_{2}$ be an abelian group, and define the proximity relation $\delta$ on $G_{2}$ as follows: For any sets $B_{1}$, $B_{2} \in G_{2}$, we say $B_{1} \ \delta \ B_{2}$ if and only if $B_{1}^{-1} \cdot B_{2}$ is of finite order in $G_{2}$. The map $\mu_{1} : G_{2} \times G_{2} \rightarrow G_{2}$, $\mu_{1}(g_{2},g_{2}^{'}) = g_{2}\cdot g_{2}^{'}$, is pcont: Let $B_{1} \times B_{2}$, $C_{1} \times C_{2} \in G_{2} \times G_{2}$ be any subsets satisfying that $B_{1} \times B_{2}$ is near $C_{1} \times C_{2}$. Then we have that $B_{1} \ \delta \ C_{1}$ and $B_{2} \ \delta \ C_{2}$, i.e., $B_{1}^{-1}C_{1}$ and $B_{2}^{-1}C_{2}$ are of finite order in $G_{2}$, respectively. Therefore, $(B_{2}^{-1}C_{2})\cdot(B_{1}^{-1}C_{1})$ is of finite order in $G_{2}$. Since $G_{2}$ is an abelian group, $(B_{1}B_{2})^{-1}C_{1}C_{2}$ is of finite order in $G_{2}$, which means that $B_{1}\cdot B_{2}$ is near $C_{1} \cdot C_{2}$ in $G_{2} \times G_{2}$. Moreover, the map $\mu_{2} : G_{2} \times G_{2} \rightarrow G_{2}$, $\mu_{2}(g_{2}) = g_{2}^{-1}$, is pcont: Let $B_{1}$, $B_{2} \in G$ with $B_{1} \ \delta \ B_{2}$. Then $B_{1}^{-1}B$ is of finite order $n$ in $G_{2}$, i.e., the $n-$times product $(B_{1}^{-1}B_{2})(B_{1}^{-1}B_{2})\cdots(B_{1}^{-1}B_{2})$ is the identity $e_{G_{2}}$ of $G_{2}$. Since $G_{2}$ is abelian, it follows that
	\begin{eqnarray*}
		(B_{1}^{-1}B_{2})(B_{1}^{-1}B_{2})\cdots(B_{1}^{-1}B_{2}) = e_{G_{2}} &\Rightarrow& (B_{2}B_{1}^{-1})(B_{2}B_{1}^{-1})\cdots(B_{2}B_{1}^{-1}) = e_{G_{2}} \\
		&\Rightarrow& B_{2}B_{1}^{-1}B_{2}B_{1}^{-1}\cdots B_{2}B_{1}^{-1} = e_{G_{2}}\\
		&\Rightarrow& e_{G_{2}} = B_{1}B_{2}^{-1}B_{1}B_{2}^{-1}\cdots B_{1}B_{2}^{-1}\\
		&\Rightarrow& (B_{1}B_{2}^{-1})(B_{1}B_{2}^{-1})\cdots(B_{1}B_{2}^{-1}) = e_{G_{2}}. 
	\end{eqnarray*}
	Hence, $B_{1}B_{2}^{-1}$ is of finite order $n$ in $G_{2}$, namely, $B_{1}^{-1} \ \delta \ B_{2}^{-1}$ in $G_{2}$.
	Finally, $(G_{2},\delta,\cdot)$ forms a proximal group.
\end{example}

\begin{theorem}\label{teo1}
	Let $(G_{1},\delta,\cdot)$ be a proximal group and $x_{1} \in G_{1}$. Then 
	\begin{eqnarray*}
		L_{x_{1}} : G_{1} \rightarrow G_{1},
	\end{eqnarray*} 
    defined by $L_{x_{1}}(g_{1}) = x_{1}g_{1}$, and
    \begin{eqnarray*}
    	R_{x_{1}} : G_{1} \rightarrow G_{1},
    \end{eqnarray*} 
    defined by $R_{x_{1}}(g_{1}) = g_{1}x_{1}$, are proximal isomorphisms.
\end{theorem}

\begin{proof}
	First, we shall show that $L_{x_{1}}$ is a proximal isomorphism. Define a map \[\nu_{x_{1}} : G_{1} \rightarrow G_{1} \times G_{1}\] by $\nu_{x_{1}}(y_{1}) = (x_{1},y_{1})$. $B_{1} \ \delta \ B_{2}$ implies that $(\{x_{1}\},B_{1}) \ \delta^{'} \ (\{x_{1}\},B_{2})$ for any subsets $B_{1}$, $B_{2} \in G_{1}$, where $\delta^{'}$ is a proximity on $G_{1} \times G_{1}$. It follows that $\nu_{x_{1}}$ is pcont. Since $G_{1}$ is a proximal group, $\mu_{1}$ is pcont. Therefore, $L_{x_{1}}$ is pcont because $L_{x_{1}} = \mu_{1} \circ \nu_{x_{1}}$. With the same method, the proximal continuity of $L_{x_{1}}^{-1}$ can be easily shown by considering the fact \[L_{x_{1}}^{-1} = L_{x_{1}^{-1}}.\] Hence, $L_{x_{1}}$ is a proximal isomorphism. The argument is similar for $R_{x_{1}}$.
\end{proof}

\quad We say that a group $G_{1}$ has an invertible subset property with respect to a subset $B_{1} \subseteq G_{1}$ provided that $B_{1}\cdot B_{1}^{-1} = \{e_{G_{1}}\}$ and $B_{1}^{-1}\cdot B_{1} = \{e_{G_{1}}\}$. Note that a group always has an invertible subset property with respect to its one-point subsets.

\begin{lemma}\label{lem1}
	Let $\delta$ be a proximity and $\cdot$ a group operation on a set $G_{1}$, respectively. Assume that
	$\mu_{1} : G_{1} \times G_{1} \rightarrow G_{1}$, $\mu_{1}(g_{1}\cdot g_{1}^{'}) = g_{1}\cdot g_{1}^{'}$,
	is pcont and $G_{1}$ has the invertible subset property with respect to any subset of it. Then $(G_{1},\delta,\cdot)$ is a proximal group.
\end{lemma}

\begin{proof}
	We shall show that $\mu_{2} : G_{1} \rightarrow G_{1}$, $\mu_{2}(x_{1}) = x_{1}^{-1}$ is a pcont map. Let $B_{1}$, $B_{2} \subseteq G_{1}$ with $B_{1} \ \delta \ B_{2}$. Then we have $(B_{1}^{-1}\cdot B_{1}) \ \delta \ (B_{1}^{-1}\cdot B_{2})$. Since $B_{1}$ is invertible, $\{e_{G_{1}}\} \ \delta \ (B_{1}^{-1}\cdot B_{2})$. It follows that $B_{2}^{-1} \ \delta \ (B_{1}^{-1}\cdot(B_{2}\cdot B_{2}^{-1}))$. Therefore, we get $B_{2}^{-1} \ \delta \ B_{1}^{-1}$ because $B_{2}$ is invertible. This proves that $\mu_{2}$ is pcont. 
\end{proof}

\begin{theorem}\label{teo2}
	Let $\delta$ be a proximity and $\cdot$ a group operation on a set $G_{1}$, respectively. Assume that $L_{x_{1}}$ and $R_{x_{1}}$ in Theorem \ref{teo1} are pcont, and $G_{1}$ has the invertible subset property with respect to any subset of it. If $\delta$ admits the transitivity property, i.e., $B_{1} \ \delta \ B_{2}$ and $B_{2} \ \delta \ B_{3}$ imply that $B_{1} \ \delta \ B_{3}$ for any subsets $B_{1}$, $B_{2}$, $B_{3} \subseteq G_{1}$, then $(G_{1},\delta,\cdot)$ is a proximal group.
\end{theorem}

\begin{proof}
	It is enough to show that $\mu_{1}$ in Definition \ref{def1} is pcont by Lemma \ref{lem1}. Let $B_{1} \times B_{2} \ \delta \ C_{1} \times C_{2}$ in $G_{1} \times G_{1}$. Then $B_{1} \ \delta \ C_{1}$ and $B_{2} \ \delta \ C_{2}$. $B_{1} \ \delta \ C_{1}$ and the proximal continuity of $R_{x_{1}}$ imply that $B_{1}\cdot B_{2}$ is near $C_{1}\cdot B_{2}$ in $G_{1} \times G_{1}$. Similarly, $B_{2} \ \delta \ C_{2}$ and the proximal continuity of $L_{x_{1}}$ imply that $C_{1} \cdot B_{2}$ is near $C_{1} \cdot C_{2}$ in $G_{1} \times G_{1}$. The transitivity property of $\delta$ says that $B_{1}\cdot B_{2}$ is near $C_{1}\cdot C_{2}$ in $G_{1} \times G_{1}$. It follows that $\mu_{1}(B_{1} \times B_{2}) \ \delta \ \mu_{1}(C_{1} \times C_{2})$, which means that $\mu_{1}$ is pcont. 
\end{proof}

\quad In Theorem \ref{teo2}, if we specifically choose the Lodato proximity $\delta^{'}$ on $G_{1}$, we need a slightly weaker condition instead of the transitivity property as follows:

\begin{corollary}
	Let $\delta^{'}$ be a Lodato proximity and $\cdot$ a group operation on a set $G_{1}$, respectively. Assume that $L_{x_{1}}$ and $R_{x_{1}}$ in Theorem \ref{teo1} are pcont, and $G_{1}$ has the invertible subset property with respect to any subset of it. If $\delta^{'}$ admits that $B_{1} \ \delta \ B_{2}$ imply $\{x_{1}\} \ \delta \ B_{2}$ for all $x_{1} \in B_{1}$, then $(G_{1},\delta^{'},\cdot)$ is a proximal group.
\end{corollary}

\begin{proof}
	Assume that $B_{1}\cdot B_{2}$ is near $C_{1}\cdot B_{2}$ and $C_{1} \cdot B_{2}$ is near $C_{1}\cdot C_{2}$ for any $B_{1} \times B_{2}$ and $C_{1} \times C_{2}$ in $G_{1} \times G_{1}$. Since $C_{1} \cdot B_{2}$ is near $C_{1} \cdot C_{2}$, it follows that $\{c_{1}b_{2}\}$ is near $C_{1} \cdot C_{2}$ for all $c_{1}b_{2} \in C_{1} \cdot B_{2}$. Therefore, we get $B_{1} \cdot B_{2}$ is near $C_{1} \cdot C_{2}$, which proves that $\mu_{1}$ is pcont. 
\end{proof}

\begin{proposition}\label{prop1}
	Let $(G_{1},\delta,\cdot)$ be a proximal group and $H_{1}$ a subgroup of $G_{1}$. Then $(H_{1},\delta_{H_{1}},\cdot)$ is a proximal group.
\end{proposition}

\begin{proof}
	Since $(G_{1},\delta,\cdot)$ is a proximal group, 
	\begin{eqnarray*}
		\mu_{1} : G_{1} \times G_{1} \rightarrow G_{1}, \ \ \ \mu_{1}(g_{1},g_{1}^{'}) = g_{1}\cdot g_{1}^{'}
	\end{eqnarray*}
    and
    \begin{eqnarray*}
    	\mu_{2} : G_{1} \rightarrow G_{1}, \ \ \ \mu_{2}(g_{1}) = g_{1}^{-1}
    \end{eqnarray*}
    are pcont. Then the restrictions \[\mu_{1}|_{H_{1} \times H_{1}} : H_{1} \times H_{1} \rightarrow H_{1}\] defined by $\mu_{1}|_{H_{1} \times H_{1}}(h_{1},h_{1}^{'}) = h_{1}\cdot h_{1}^{'}$ and \[\mu_{2}|_{H_{1}} : H_{1} \rightarrow H_{1}\] defined by $\mu_{2}|_{H_{1}}(h_{1}) = h_{1}^{-1}$ are pcont, respectively. This shows that $(H_{1},\delta_{H_{1}},\cdot)$ is a proximal group.
\end{proof}

\quad Note that, in Proposition \ref{prop1}, $(H_{1},\delta_{H_{1}},\cdot)$ is said to be a proximal subgroup of $(G_{1},\delta,\cdot)$. As an example, $\mathbb{R}^{+}$ is a proximal subgroup of $\mathbb{R}-\{0\}$ in Example \ref{exm1}.

\begin{proposition}\label{prop2}
	Given any proximal groups $(G_{1},\delta_{1},\cdot_{1})$ and $(G_{2},\delta_{2},\cdot_{2})$, their cartesian product $G_{1} \times G_{2}$ is also a proximal group.
\end{proposition}

\begin{proof}
	Since $G_{1}$ is a proximal group with a proximity $\delta_{1}$ and a group operation $\cdot_{1}$, we have that
	\begin{eqnarray*}
		\mu_{1} : G_{1} \times G_{1} \rightarrow G_{1}, \ \mu_{1}(g_{1},g_{1}^{'}) = g_{1}\cdot_{1} g_{1}^{'} \hspace*{0.5cm} \text{and} \hspace*{0.5cm} \mu_{2} : G_{1} \rightarrow G_{1}, \ \mu_{2}(g_{1}) = g_{1}^{-1}
	\end{eqnarray*}
    are pcont. Similarly, from the proximal group construction of $G_{2}$, we have that
    \begin{eqnarray*}
    	\mu_{1}^{'} : G_{2} \times G_{2} \rightarrow G_{2}, \ \mu_{1}^{'}(g_{2},g_{2}^{'}) = g_{2}\cdot_{2} g_{2}^{'} \hspace*{0.5cm} \text{and} \hspace*{0.5cm} \mu_{2}^{'} : G_{2} \rightarrow G_{2}, \ \mu_{2}^{'}(g_{2}) = g_{2}^{-1}
    \end{eqnarray*}
    are pcont. Define two maps
    \begin{eqnarray*}
    	\mu_{3} : (G_{1} \times G_{2})\times(G_{1} \times G_{2}) \rightarrow G_{1} \times G_{2}
    \end{eqnarray*}
    and
    \begin{eqnarray*}
    	\mu_{4} : G_{1} \times G_{2} \rightarrow G_{1} \times G_{2}
    \end{eqnarray*}
    by $\mu_{3}((g_{1},g_{2}),(g_{1}^{'},g_{2}^{'})) = (\mu_{1}(g_{1},g_{1}^{'}), \mu_{2}(g_{2},g_{2}^{'}))$ and $\mu_{4}(g_{1},g_{2}) = (\mu_{1}^{'}(g_{1}),\mu_{2}^{'}(g_{2}))$, respectively. Then $\mu_{3}$ and $\mu_{4}$ are pcont from the definition of cartesian product proximity. Thus, $G_{1} \times G_{2}$ is a proximal group having the product proximity $\delta_{1} \times \delta_{2}$ on itself.
\end{proof}

\begin{definition}
	Let $(G_{1},\delta_{1},\cdot_{1})$ and $(G_{2},\delta_{2},\cdot_{2})$ be any proximal groups. Then $\eta : G_{1} \rightarrow G_{2}$ is called a homomorphism of proximal groups provided that it is pcont group homomorphism. Furthermore, $\eta$ is called an isomorphism of proximal groups if it is a group isomorphism and also a proximal isomorphism. 
\end{definition}

\begin{example}
	Consider the antipodal map $\eta : (\mathbb{R},\delta,+) \rightarrow (\mathbb{R},\delta,+)$, $\eta(x) = -x$, where $\delta$ is given in Example \ref{exm1}, and $+$ is the usual additive group operation. For $B_{1}$, $B_{2} \in \mathbb{R}$, $B_{1} \ \delta \ B_{2}$ means that $B_{1} \cap B_{2} \neq \emptyset$. Then there exists $r \in \mathbb{R}$ such that $r$ belongs to both $B_{1}$ and $B_{2}$. Since $(\mathbb{R},+)$ is a group, $r$ has an inverse $-r$ in $\mathbb{R}$. It follows that $-r$ belongs to both $-B_{1}$ and $-B_{2}$, i.e., $(-B_{1}) \cap (-B_{2}) \neq \emptyset$. Therefore, \[\eta(B_{1}) = (-B_{1}) \ \delta \ (-B_{2}) = \eta(B_{2}).\] Thus, $\eta$ is pcont. Similarly, it can be easily shown that $\eta^{-1}$ is a pcont map. On the other hand, we observe
	\begin{eqnarray*}
		\eta(B_{1} + B_{2}) = -(B_{1} + B_{2}) = -B_{1} + (-B_{2}) = \eta(B_{1}) + \eta(B_{2}), 
	\end{eqnarray*} 
    which shows that $\eta$ is a group homomorphism. As a consequence, $\eta$ is a proximally group isomorphism.
\end{example}

\begin{theorem}
	Let $\eta : (G_{1},\delta_{1},\cdot_{1}) \rightarrow (G_{2},\delta_{2},\cdot_{2})$ be a group homomorphism between two proximal groups $G_{1}$ and $G_{2}$ such that they have the invertible subset property with respect to any subset of them. Then $\eta$ is a proximal homomorphism provided that $B_{1} \ \delta_{1} \ \{e_{G_{1}}\}$ implies that $\eta(B_{1}) \ \delta_{2} \ \{e_{G_{2}}\}$.
\end{theorem}

\begin{proof}
	Let $B_{1} \ \delta_{1} \ B_{2}$ for any $B_{1}$, $B_{2} \in G_{1}$. Then $(B_{1}B_{2}^{-1}) \ \delta_{1} \ (B_{2}B_{2}^{-1}) = \{e_{G_{1}}\}$ because $G_{1}$ is a proximal group. It follows that $\eta(B_{1}B_{2}^{-1}) \ \delta_{2} \ \eta(\{e_{G_{2}}\})$. Since $\eta$ is a group homomorphism, $\eta(B_{1}B_{2}^{-1}) = \eta(B_{1})\eta(B_{2})^{-1}$ and $\eta(\{e_{G_{1}}\}) = \{e_{G_{2}}\}$. Therefore, we get $(\eta(B_{1})\eta(B_{2})^{-1}) \ \delta_{2} \ \{e_{G_{2}}\}$. $G_{2}$ is a proximal group, so we find $\eta(B_{1}) \ \delta_{2} \ \eta(B_{2})$, which shows that $\eta$ is pcont.
\end{proof}

\begin{proposition}
	\textbf{i)} The First Isomorphism Theorem does not work for proximal groups.
	
	\textbf{ii)} The Second Isomorphism Theorem does not work for proximal groups.
	
	\textbf{iii)} The Third Isomorphism Theorem holds for proximal groups. 
\end{proposition}

\begin{proof}
	\textbf{i)} Let $(\mathbb{R},\delta_{1},+)$ and $(\mathbb{R},\delta_{2},+)$ be two proximal groups, where $\delta_{1}$ is the discrete proximity and $\delta_{2}$ is given by $B_{1} \ \delta_{2} \ B_{2} \Leftrightarrow D(B_{1},B_{2}) = 0$. The identity map
	\begin{eqnarray*}
		1 : (\mathbb{R},\delta_{1},+) \rightarrow (\mathbb{R},\delta_{2},+)
	\end{eqnarray*}
    is both a pcont map and a group homomorphism. Also, the identity map is surjective and $Ker(1)$ only consists of the identity element of $(\mathbb{R},+)$. However, $\mathbb{R} / Ker(1)$, that is proximally isomorphic to $\mathbb{R}$ as proximal groups, with the discrete proximity is not proximally isomorphic to $\mathbb{R}$ with the proximity $\delta_{2}$: Since $D(B_{1},B_{2}) = 0$ does not always imply $B_{1} \cap B_{2} \neq \emptyset$ for any $B_{1}$, $B_{2} \subseteq \mathbb{R}$, the inverse of the identity map $1$ is not pcont. 
    
    \textbf{ii)} Let $G_{1} = (\mathbb{R},\delta,+)$ be a proximal group with its subgroup $H_{1} = \{2s \ | \ s \in \mathbb{Q}^{c}\}$ and its normal subgroup $N_{1} = \mathbb{Z}$, where $\delta$ is the proximity $\delta_{2}$ given in i). Then the intersection of $H_{1}$ and $N_{1}$ is empty, which follows that, as proximal groups, $H_{1}$ is proximally isomorphic to $H_{1} / (H_{1} \cap N_{1})$ as proximal groups. Since $d(B_{1},B_{2}) > 0$ for all $B_{1}$, $B_{2} \subseteq H_{1}$, $H_{1}$ must have the discrete proximity. On the other hand, we have that $(H_{1}+N_{1})^{\delta} = \mathbb{R}$. Therefore, $[(H_{1}+N_{1})/N_{1}]^{\delta} = \mathbb{R}/\mathbb{Z}$, which means that $(H_{1}+N_{1})/N_{1}$ cannot have the discrete proximity. Consequently, the map \[(H_{1}+N_{1})/N_{1} \rightarrow H_{1} / (H_{1} \cap N_{1})\] is not a group isomorphism of proximal groups.
    
    \textbf{iii)} The proof is similar to the case of topological groups.  
\end{proof}

\quad Note that a continuous map need not be pcont. However, a pcont map is always continuous with respect to compatible topologies. Hence, given a proximal group $(Y,\delta,\cdot)$, we have that $(Y,\tau(\delta),\cdot)$ is a topological group since pcont maps $\mu_{1}$ and $\mu_{2}$ in Definition \ref{def1} are also continuous maps with respect to the compatible topologies. We say that $(Y,\tau(\delta),\cdot)$ is a topological group induced by $\delta$.

\begin{theorem}\label{teo4}
	Let $(G_{1},\delta,\cdot)$ be a proximal group. Then $G_{1}$ admits an Hausdorff topological group if and only if $\{e_{G_{1}}\} \ \delta \ B_{1}$ implies $B_{1} = \{e_{G_{1}}\}$ for $B_{1} \subseteq G_{1}$.
\end{theorem}

\begin{proof}
	The assertion is clear from the fact that, for a topological group $G_{1}$, it is Hausdorff if and only if $\{e_{G_{1}}\}$ is closed.
\end{proof}

\quad In a topological group, the axioms $T_{0}$, $T_{1}$, and $T_{2}$ (Hausdorffness) coincide. This means that one can consider the topological group $(G_{1},\delta(\tau),\cdot)$ as $T_{0}$ or $T_{1}$ instead of $T_{2}$ in Theorem \ref{teo4}.

\section{Descriptive Proximal Groups}
\label{sec:3}

\begin{definition}
	For a descriptive proximity $\delta_{\Phi}$ and a group operation $\cdot$ on a set $G_{1}$,  $(G_{1},\delta_{\Phi},\cdot)$ is said to be a descriptive proximal group when
	\begin{eqnarray*}
		\mu_{1} : G_{1} \times G_{1} \rightarrow G_{1},
	\end{eqnarray*} 
	defined by $\mu_{1}(g_{1},g_{1}^{'}) = g_{1}\cdot g_{1}^{'}$ for any $g_{1}$, $g_{1}^{'} \in G$,
	and
	\begin{eqnarray*}
		\mu_{2} : G_{1} \rightarrow G_{1},
	\end{eqnarray*} 
	defined by $\mu_{2}(g_{1}) = g_{1}^{-1}$ for any $g_{1} \in G_{1}$, are dpcont maps.
\end{definition}

\begin{example}
	Let $X$ be a set shown in Figure \ref{fig:1}, which consists of three boxes $A$, $B$, and $C$. Consider $G$ as the set of all dpcont paths on $X$. For any dpcont paths $\gamma_{1}$, $\gamma_{2}$ in $G$ with $\gamma_{1}(1) = \gamma_{2}(0)$, a group operation $\ast$ on $G$ is defined by \[\gamma_{1} \ast \gamma_{2}(s) = \begin{cases}
		\gamma_{1}(2s), & 0 \leq s \leq 1/2 \\
		\gamma_{2}(2s-1), & 1/2 \leq s \leq 1.
	\end{cases}\] 
    Note that $\gamma_{i}(s) \in \{A,B,C\}$ for each $i = 1,2$. When one considers $\Phi$ as a probe function that determines any descriptive proximal path by the order of the box names of that path, $\delta_{\Phi}$ is a descriptive proximity on $G$. Consider the map \linebreak$\mu_{1} : G \times G \rightarrow G$ given by $\mu_{1}(\gamma_{1},\gamma_{2}) = \gamma_{1} \ast \gamma_{2}$. For any $(\gamma_{1},\gamma_{2})$, $(\gamma_{3},\gamma_{4}) \in G \times G$, the fact $(\gamma_{1},\gamma_{2})$ is descriptively near $(\gamma_{3},\gamma_{4})$ implies that $\gamma_{1} \ \delta_{\Phi} \ \gamma_{3}$ and $\gamma_{2} \ \delta_{\Phi} \ \gamma_{4}$. Then, for all $s_{1}$, $s_{2} \in [0,1]$, we have that $\gamma_{1}(s_{1}) = \gamma_{3}(s_{1})$ and $\gamma_{2}(s_{2}) = \gamma_{4}(s_{2})$. It follows that $\gamma_{1} \ast \gamma_{2} = \gamma_{3} \ast \gamma_{4}$, which means that $\gamma_{1} \ast \gamma_{2}$ is descriptively near $\gamma_{3} \ast \gamma_{4}$. Hence, $\mu_{1}$ is dpcont. Moreover, the map $\mu_{2} : G \rightarrow G$, defined by $\mu_{2}(\gamma) = \gamma^{-1}$, is dpcont. Indeed, for any $\gamma_{1}$, $\gamma_{2} \in G$, $\gamma_{1} \ \delta_{\Phi} \ \gamma_{2}$ implies that $\gamma_{1}(s) = \gamma_{2}(s)$ for all $s \in [0,1]$. Therefore, $\gamma_{1}^{-1}(s) = \gamma_{2}^{-1}(s)$ for all $s \in [0,1]$. This shows that $\gamma_{1}^{-1} \ \delta_{\Phi} \ \gamma_{2}^{-1}$, and finally, $(G,\delta_{\Phi},\circ)$ is a descriptive proximal group.
\end{example}

\begin{figure}[h]
	\centering
	\includegraphics[width=1.00\textwidth]{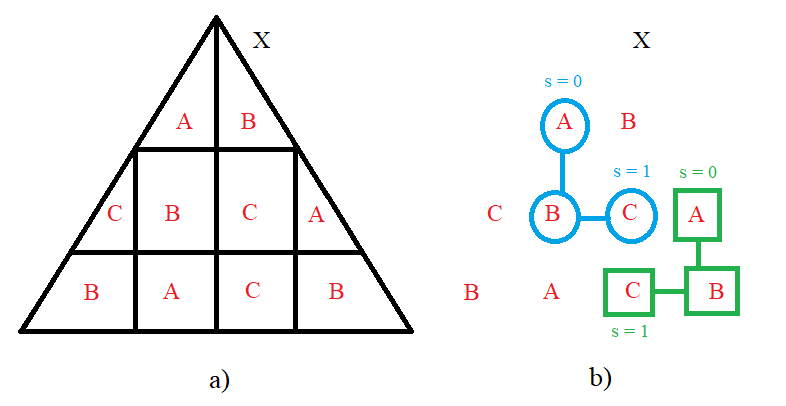}
	\caption{\textbf{a)} The picture $X$ consists of boxes $A$, $B$, and $C$. \\ \textbf{b)} The blue path ABC is descriptively near the green path. Orders of the paths are equal to each other.} 
	\label{fig:1}
\end{figure}
\newpage
\begin{theorem}
	Let $(G_{1},\delta_{\Phi},\cdot)$ be a descriptive proximal group and $x_{1} \in G_{1}$. Then 
	\begin{eqnarray*}
		L_{x_{1}} : G_{1} \rightarrow G_{1},
	\end{eqnarray*} 
	defined by $L_{x_{1}}(g_{1}) = x_{1}g_{1}$, and
	\begin{eqnarray*}
		R_{x_{1}} : G_{1} \rightarrow G_{1},
	\end{eqnarray*} 
	defined by $R_{x_{1}}(g_{1}) = g_{1}x_{1}$, are descriptive proximal isomorphisms.
\end{theorem}

\begin{proof}
	The proof is parallel with Theorem \ref{teo1} since the composition of dpcont maps is again dpcont.
\end{proof}

\begin{remark}
	For a descriptive proximal group $(G_{1},\delta_{\Phi},\cdot)$ and a subgroup $H_{1}$ of $G_{1}$, we say that $(H_{1},\delta_{\Phi_{H_{1}}},\cdot)$ is a descriptive proximal subgroup of $(G_{1},\delta_{\Phi},\cdot)$.
\end{remark} 

\begin{example}\label{exm2}
	Consider the additive group $(\mathbb{R},+)$. Let $\Phi = \{\phi_{1},\phi_{2}\}$ be the probe function such that $\phi_{1} : \mathbb{R} \rightarrow \mathbb{R}$ and $\phi_{2} : \mathbb{R} \rightarrow \mathbb{R}$ are defined by $\phi_{1}(y) = trunc(y)$ and $\phi_{2}(y) = \begin{cases}
		y, & y \in \mathbb{Z}^{c} \\
		y+0.3, & y \in \mathbb{Z}
	\end{cases}$ for any $y \in \mathbb{R}$, respectively. Here the function $trunc(y)$ truncates $y \in \mathbb{R}$ to an integer by removing the fractional part of the number. For example, $trunc(-3.4) = -3$ and $trunc(5.6) = 5$. Note that $\phi_{1}(y) \in \mathbb{Z}$ and $\phi_{2}(y) \in \mathbb{Z}^{c}$. To show that \[\mu_{1} : \mathbb{R} \times \mathbb{R} \rightarrow \mathbb{R}, \ \mu_{1}(y_{1},y_{2}) = y_{1}+y_{2},\] is dpcont, we shall show that $B_{1} \times B_{2}$ is descriptively near $C_{1} \times C_{2}$ implies that $(B_{1}+B_{2}) \ \delta_{\Phi} \ (C_{1}+C_{2})$ for any $B_{1} \times B_{2}$, $C_{1} \times C_{2} \in \mathbb{R} \times \mathbb{R}$. Since $B_{1} \times B_{2}$ is descriptively near $C_{1} \times C_{2}$, we have that $B_{1}$ is descriptively near $C_{1}$ and $B_{2}$ is descriptively near $C_{2}$. When we consider $B_{1} \ \delta_{\Phi} \ C_{1}$ (i.e., $\mathcal{Q}(B_{1}) \cap \mathcal{Q}(C_{1}) \neq \emptyset$), there are some cases as follows:
\begin{itemize}
	\item There exists $b_{1} \in \mathbb{R}$ such that $b_{1} \in B_{1} \cap C_{1}$.
	\item There exist $b_{1} \in \mathbb{Z}$ and $c_{1} \in \mathbb{Z}^{c}$ such that $b_{1} \in B_{1}$ and $c_{1} \in C_{1}$ with $trunc(c_{1}) = b_{1}$.
	\item There exist $b_{1} \in \mathbb{Z}$ and $c_{1} \in \mathbb{Z}^{c}$ such that $b_{1} \in B_{1}$ and $c_{1} \in C_{1}$ with $c_{1} = b_{1} + 0.3$.
\end{itemize}
    The cases are hold when we focus on $B_{2} \ \delta_{\Phi} \ D_{2}$ (i.e., $\mathcal{Q}(B_{2}) \cap \mathcal{Q}(D_{2}) \neq \emptyset$):
 \begin{itemize}
 	\item There exists $b_{2} \in \mathbb{R}$ such that $b_{2} \in B_{2} \cap C_{2}$.
 	\item There exist $b_{2} \in \mathbb{Z}$ and $c_{2} \in \mathbb{Z}^{c}$ such that $b_{2} \in B_{2}$ and $c_{2} \in C_{2}$ with $trunc(c_{2}) = b_{2}$.
 	\item There exist $b_{2} \in \mathbb{Z}$ and $c_{2} \in \mathbb{Z}^{c}$ such that $b_{2} \in B_{2}$ and $c_{2} \in C_{2}$ with $c_{2} = b_{2} + 0.3$.
 \end{itemize} 
    For all cases, we have that $\mathcal{Q}(B_{1}+B_{2}) \cap \mathcal{Q}(C_{1}+C_{2}) \neq \emptyset$. This means that $B_{1}+B_{2}$ is descriptively near $C_{1}+C_{2}$. Now, define 
    \begin{eqnarray*}
    	\mu_{2} : \mathbb{R} \rightarrow \mathbb{R}, \ \mu_{2}(y) = -y.
    \end{eqnarray*}
    Let $B_{1}$, $B_{2} \in 2^{\mathbb{R}}$ with $B_{1} \ \delta_{\Phi} \ B_{2}$. Then there are three cases again:
    \begin{itemize}
    	\item There exists $b_{1} \in \mathbb{R}$ such that $b_{1} \in B_{1} \cap B_{2}$.
    	\item There exist $b_{1} \in \mathbb{Z}$ and $b_{2} \in \mathbb{Z}^{c}$ such that $b_{1} \in B_{1}$ and $b_{2} \in B_{2}$ with $trunc(b_{2}) = b_{1}$.
    	\item There exist $b_{1} \in \mathbb{Z}$ and $b_{2} \in \mathbb{Z}^{c}$ such that $b_{1} \in B_{1}$ and $b_{2} \in B_{2}$ with $b_{2} = b_{1} + 0.3$.
    \end{itemize}
    In each case, we find that $-B_{1}$ is descriptively near $-B_{2}$:
    \begin{itemize}
    	\item If there is a real number $b_{1} \in B_{1} \cap B_{2}$, then the real number $-b_{1}$ belongs to both $-B_{1}$ and $-B_{2}$.
    	\item If $trunc(b_{2}) = b_{1}$, then $trunc(-b_{2}) = -b_{1}$. 
    	\item If $b_{2} = b_{1} + 0.3$, then $-b_{1} = -b_{2} + 0.3$.
    \end{itemize}
    Therefore, we observe that $-b_{1} \in \mathcal{Q}(-B_{1}) \cap \mathcal{Q}(-B_{2})$ for all cases, namely that, $(-B_{1}) \ \delta_{\Phi} \ (-B_{2})$. This shows that $\mu_{2}$ is dpcont. Hence, $(\mathbb{R},\delta_{\Phi},+)$ is a descriptive proximal group. Moreover, the fact $(\mathbb{Q},+)$ is a subgroup of $(\mathbb{R},+)$ shows that $(\mathbb{Q},\delta_{\Phi_{\mathbb{Q}}},+)$ is a descriptive proximal subgroup.
\end{example}

\quad Similar to Proposition \ref{prop2}, the cartesian product of two descriptive proximal groups is a descriptive proximal group. Consider the descriptive proximal group $(\mathbb{R},\delta_{\Phi},+)$ given in Example \ref{exm2}. Then we have that $(\mathbb{R}^{2},\delta_{\Phi^{'}},+)$ is also a descriptive proximal group, where $\delta_{\Phi^{'}}$ is the cartesian product descriptive proximity $\delta_{\Phi} \times \delta_{\Phi}$.

\begin{definition}
	Let $(G_{1},\delta_{\Phi_{1}},\cdot_{1})$ and $(G_{2},\delta_{\Phi_{2}},\cdot_{2})$ be any descriptive proximal groups. Then $\eta : G_{1} \rightarrow G_{2}$ is called a homomorphism of descriptive proximal groups provided that it is dpcont group homomorphism. Furthermore, $\eta$ is called an isomorphism of descriptive proximal groups if it is a group isomorphism and also a descriptive proximal isomorphism. 
\end{definition}

\begin{example}
	Let $(Y_{1},\delta_{\Phi_{1}},+)$ and $(Y_{2},\delta_{\Phi_{2}},+)$ be any proximal groups. Then define a dpcont map
	\begin{eqnarray*}
		\nu : (Y_{1},\delta_{\Phi_{1}},+) \times (Y_{2},\delta_{\Phi_{2}},+) \rightarrow (Y_{1},\delta_{\Phi_{1}},+)
	\end{eqnarray*} 
    by $\nu(y_{1},y_{2}) = y_{1}$. $\nu$ is a homomorphism of descriptive proximal groups. Indeed, for $B_{1} \times B_{2}$, $C_{1} \times C_{2} \in 2^{Y_{1} \times Y_{2}}$,
    \begin{eqnarray*}
    	\nu[(B_{1} \times B_{2}) + (C_{1} \times C_{2})] &=& \nu[(B_{1}+C_{1}),(B_{2}+C_{2})] \\ 
    	&=& B_{1} + C_{1}\\
    	&=& \nu(B_{1} \times B_{2}) + \eta(C_{1} \times C_{2}).
    \end{eqnarray*}
    However, $\nu$ is not an isomorphism of descriptive proximal groups: For any $B_{1}$, $C_{1} \in 2^{Y_{1}}$, if $B_{1} \ \delta_{\Phi_{1}} \ C_{1}$ and $B_{2} \ \underline{\delta_{\Phi_{2}}} \ C_{2}$, then $B_{1} \times B_{2}$ cannot be descriptively near $C_{1} \times C_{2}$. 
\end{example}

\section{Conclusion}
\label{conc:5}
\quad The study of topological groups in proximity and descriptive proximity spaces marks an important step in the advancement of nearness theory, providing a fresh perspective on the interplay between algebraic and topological structures. As we continue to explore the potential of nearness-based settings, this research opens up new avenues for future investigations, promising further advancements and exciting discoveries in the fascinating realm of topological groups.

\quad It is necessary to mention an open problem on isomorphism theorems for groups. Intuitively, it can be thought that the first and second isomorphism theorem is not satisfied, but the third isomorphism theorem is satisfied, just as in proximity spaces. Making this clear with examples or proofs is to take the matter one step further. As another open problem, Lie groups setting in the theory of proximity (or descriptive proximity) can be considered. However, for this problem, first of all, the concept of a manifold and its related invariants should be studied extensively in the theory of nearness. As a result, it is very possible to obtain interesting results using the (descriptive) proximal group results on descriptive proximity theory.

\acknowledgment{The Scientific and Technological Research Council of Turkey TÜBİTAK-1002-A supported this investigation under project number 122F454.}

\end{document}